\documentclass[12pt]{article}
\usepackage{setspace}

\usepackage[leqno]{amsmath}
\usepackage{amssymb,amsfonts,xfrac,MnSymbol}
\usepackage{amsthm}
\usepackage{cite}
\usepackage{hyperref}
\usepackage[capitalise]{cleveref}
\usepackage{todonotes}
\usepackage{enumitem}
\usepackage{graphicx}
\usepackage{tikz-cd, tikz}
\usepackage{color}
\usepackage{import}
\usepackage{stmaryrd}
\usepackage{subcaption}
\usepackage{comment}

\numberwithin{equation}{section}

\newtheorem{theorem}{Theorem}[section]
\newtheorem{claim}[theorem]{Claim}
\newtheorem{proposition}[theorem]{Proposition}
\newtheorem{lemma}[theorem]{Lemma}

\newtheorem{corollary}[theorem]{Corollary}

\newtheorem*{theorem*}{Theorem}
\newtheorem*{claim*}{Claim}
\newtheorem*{proposition*}{Proposition}
\newtheorem*{lemma*}{Lemma}
\newtheorem*{corollary*}{Corollary}

\newtheorem{theoremA}{Theorem}

\theoremstyle{definition}
\newtheorem{definition}[theorem]{Definition}

\newtheorem{remark}[theorem]{Remark}

\newtheorem{example}[theorem]{Example}

\newtheorem*{definition*}{Definition}
\newtheorem*{observation*}{Observation}
\newtheorem*{remark*}{Remark}
\newtheorem*{example*}{Example}
\newtheorem*{question*}{Question}
\newtheorem*{exercise*}{Exercise}
\newtheorem*{fact*}{Fact}
\newtheorem*{notation*}{Notation}

\newcommand{\bbN}{\mathbb{N}}

\newcommand{\bbQ}{\mathbb{Q}}
\newcommand{\bbR}{\mathbb{R}}

\newcommand{\bbZ}{\mathbb{Z}}

\newcommand{\bfc}{\mathbf{c}}

\newcommand{\bfs}{\mathbf{s}}
\newcommand{\bft}{\mathbf{t}}

\newcommand{\bfw}{\mathbf{w}}

\newcommand{\calA}{\mathcal{A}}
\newcommand{\calB}{\mathcal{B}}

\newcommand{\calF}{\mathcal{F}}

\newcommand{\calN}{\mathcal{N}}

\newcommand{\actson}{\curvearrowright}

\newcommand{\cechH}{\check{H}}

\newcommand{\ii}{^{-1}}

\newcommand{\gen}[1]{\left< #1 \right>}

\newcommand{\nnN}{\mathbb{N}_{\ge0}}
\newcommand{\geod}[2]{{\llbracket {#1} , {#2}\rrbracket}}
\newcommand{\adelta}{{\boldsymbol\delta}}

\DeclareMathOperator{\Aut}{Aut}

\DeclareMathOperator{\Cay}{Cay}

\DeclareMathOperator{\Out}{Out}

\DeclareMathOperator{\im}{Im}
\DeclareMathOperator{\Vol}{Vol}
\DeclareMathOperator{\rank}{rank}
\DeclareMathOperator{\Cx}{C}

\DeclareMathOperator{\supp}{supp}
\DeclareMathOperator{\df}{def}

\newcommand{\ie}{\textit{i.e.} }

\title{Finite index rigidity of hyperbolic groups}
\author{Nir Lazarovich\thanks{Supported by the Israel Science Foundation (grant no. 1562/19).}}
\date{}

\begin{document}
\spacing{1}
\maketitle

\begin{abstract}
    We prove that the topological complexity of a finite index subgroup of a hyperbolic group is linear in its index. 
    This follows from a more general result relating the size of the quotient of a free cocompact action of a hyperbolic group on a graph to the minimal number of cells in a simplicial classifying space for the group.
    As a corollary we prove that any two isomorphic finite-index subgroups of a non-elementary hyperbolic group have the same index.
\end{abstract}

\section{Introduction}

We say that a group $G$ is \emph{finite index rigid} (or \emph{satisfies the volume condition} \cite{bridson2010cofinitely}) if it does not contain isomorphic finite index subgroups of different indices. 
For example, the infinite cyclic group $\bbZ$ is not finite index rigid (\textit{e.g.} $2\bbZ\simeq 3\bbZ$). 
In contrast, the following families of groups are finite index rigid: finite groups; groups with non-zero (rational) Euler characteristic (\textit{e.g.} non-abelian free groups and hyperbolic surface groups); groups with some non-zero $\ell^2$-betti number; lattices in simple Lie groups  \cite{mostow1968quasi}; groups with deficiency $\ge 2$ \cite{reznikov1996volumes}; most 3-manifold groups \cite{wang1994covering,wang1999covering}; groups with infinitely many ends \cite{sykiotis2018complexity} (cf. also \S6 of \cite{lazarovich2021volume}); and groups which have a JSJ decomposition with a maximal hanging Fuchsian group \cite{sykiotis2018complexity}.

In all of the above cases, the proofs rely on the existence of a non-zero numerical invariant which depends linearly on the index of a finite index subgroup. 
In this paper we find such an invariant for hyperbolic groups:


\begin{theoremA}\label{thm: isomorphic finite index}
Every non-elementary hyperbolic group is finite index rigid.
\end{theoremA}

To put this theorem in context, note that Sela \cite{sela1997structure} proved that one-ended hyperbolic groups are co-Hopfian, i.e. they are not isomorphic to their proper subgroups.  
On the other hand, Stark-Woodhouse \cite{stark2018hyperbolic} found an example of a one-ended hyperbolic group with two isomorphic subgroups of different indices -- a finite-index subgroup which is isomorphic to an infinite-index subgroup.
The theorem above shows that it is impossible to find such an example in which both subgroups have finite index.


As an immediate corollary we obtain:
\begin{corollary}\label{cor: free by cyclic}
Let $\varphi \in \Out(F_d)$ be an atoroidal outer automorphism of the free group $F_d$, then $F_d \rtimes \gen{\varphi^m} \simeq F_d \rtimes \gen{\varphi^n}$ if and only if $|m|=|n|$.
\end{corollary}

\begin{proof}
The group $G=F_d \rtimes \gen{\varphi}$ is hyperbolic by \cite{brinkmann2000hyperbolic}. For every $k$ the group $F_d \rtimes \gen{\varphi^k}$ is a subgroup of $G$ of index $|k|$. The corollary now follows by \cref{thm: isomorphic finite index}.
\end{proof}

As mentioned above, the common strategy for proving finite index rigidity is to associate some ``complexity'' to a group and to show that it is proportional to its index. 
To elucidate this, consider the case of a non-abelian free group $G=F_d$ and consider a finite index subgroup $A\le G$. We have 
\begin{equation}\label{rank and index} 
 1-\rank(A)=\chi(A) = [G:A]\chi(G)
\end{equation}
where $\chi(G)$ denotes the Euler characteristic of $G$.
On the left hand side of \cref{rank and index}, we have a number that depends only on the isomorphism type of $A$ (and can be thought of measuring the ``complexity'' of the free group $A$), while on the right we have a number that depends linearly on the index $[G:A]$ (and on $G$). 
It easily follows from \cref{rank and index} that if two finite index subgroups $A,B\le G$ are isomorphic then $[G:A]=[G:B]$.

The proof of \cref{thm: isomorphic finite index} relies on a similar relation between complexity and index (or more generally covolume) in a discrete hyperbolic setting.
For a compact simplicial complex $\bar X$, we denote by $\Vol(\bar X)$ the total number of cells of $\bar X$, and by $\Vol_i(\bar X)$ the total number of cells in its $i$-skeleton.
For $i\le m$ and a group $G$ define $\Cx_{i,m}(G)$ to be the minimal $\Vol_i(\bar L)$ where $\bar L$ runs over all $m$-dimensional simplicial complexes such that $\pi_1(\bar L)=G$ and whose universal cover is $(m-1)$-connected.
(If such a complex $\bar L$ does not exist then $\Cx_{i,m}(G) = \infty$.)
The value of $\Cx_{i,m}(G)$ can be thought of as measuring the topological complexity of $G$, similar to the Kneser complexity $k(M)$ measuring the minimal triangulation of a manifold $M$. 
For $m=1$ this invariant is related to the rank $\Cx_{1,1}(G)\asymp \rank(G)$ (see \S\ref{sec: notation} for the definition of the notation $\asymp$). For $i=m=2$, if $G$ is one-ended, $\Cx_{2,2}(G)\asymp T(G)$, where $T(G)$ is Delzant's T-invariant \cite{delzant1996decomposition} defined by $$T(G) = \min \left\{\;\; \sum_{r\in R} \max\{|r|-2,0\} \;\;:\;\; \text{for all presentations } G=\gen{S \;|\; R}\;\;\right\}.$$

Our main result is the following relation between covolume and complexity.
\begin{theoremA}\label{thm: volume vs complexity}
Let $X$ be a locally finite one-ended hyperbolic graph. Then, there exists $m\in \bbN$ so that if $G$ acts on $X$ freely and cocompactly then $$ \Vol(X/G) \; \asymp_X \; \Cx_{2,m}(G) .$$
\end{theoremA}

In fact, it suffices to take $m=\dim(\partial X)+1$ where $\partial X$ is the Gromov boundary and $\dim$ is its topological covering dimension. In particular, if $\dim(\partial X)=1$ we get a similar inequality for Delzant's $T$ invariant.

In the smooth setting, \textit{i.e.} when $X$ is a non-positively curved Riemannian manifold,  relations between covolume of lattices and measurements of complexity were extensively studied \cite{bader2020homology,ballmann1985manifolds,belolipetsky2010counting,cooper1999volume,delzant1996decomposition,delzant2013complexity,gelander2011volume,gelander2004homotopy,gelander2019minimal,gelander2021bounds,gromov1982volume,thurston1979geometry,reznikov1996volumes}.
We note that in the discrete setting of \cref{thm: volume vs complexity}, the inequality $\Cx_{i,m}(G) \ll_X \Vol(X/G)$ is almost immediate from the contractibility of the Rips complex and the definition of $\Cx_{i,m}(G)$ (see \S\ref{sec: notation} for the definition of the notation $\ll$). Therefore, most of the effort is in proving the other inequality, namely $\Vol(X/G) \ll_X \Cx_{2,m}(G)$. Loosely, this inequality states that the more volume $X/G$ has the more complicated the group $G$ should be as an abstract group.
To prove this inequality we start with some compact complex $\bar L$ that realizes $\Cx_{2,m}(G)$ (for large enough $m$ that depends only on $X$). In \S\ref{sec: global stability}--\S\ref{sec: resolution}, drawing on ideas of Delzant \cite{delzant1995image} and Rips-Sela \cite{rips1995canonical}, we use Mineyev's rational bicombing \cite{mineyev2001straightening} on $X$ to obtain a singular weighted pattern $\bar \calF$ on the 2-skeleton of $\bar L$. In \S\ref{sec: accessibility}, we prove that the total weight of the pattern $\bar \calF$ is bounded above by $\Vol_2(\bar L)=\Cx_{2,m}(G)$. In \S\ref{sec: quasisurjectivity}, using cohomological methods of Bestvina-Mess \cite{bestvina1991boundary}, we show that a continuous $G$-equivariant map from the universal cover of $\bar L$ to $X$ must be quasi-surjective with a constant that depends only on $X$ (``uniform quasi-surjectivity'' \cref{prop: uniform quasi-onto}). In \S\ref{sec: main proof}, we use this to show that the total weight of the pattern $\bar \calF$ is bounded below by the volume $\Vol(X/G)$. Combining the two inequalities gives $$\Vol(X/G) \ll_X \text{ total weight of $\bar \calF$ }\ll_X \Cx_{2,m}(G).$$

As a corollary of \cref{thm: volume vs complexity} one gets the desired relation between index and complexity:

\begin{theoremA}\label{thm: index vs complexity}
Let $G$ be a one-ended hyperbolic group. Then, there exists $m\in \bbN$ such that if $H$ is a finite index subgroup of $G$ then $$ [G:H] \; \asymp_G \;  C_{2,m}(H).$$
\end{theoremA}

\begin{proof}[\cref{thm: volume vs complexity} implies \cref{thm: index vs complexity}]
Let $X=\Cay(G,S)$ be the Cayley graph of $G$ with respect to some finite generating set $S$. The result now follows by  \cref{thm: volume vs complexity} since a finite index subgroup $H$ acts freely and cocompactly on $X$ and $\Vol(X/H) \asymp_X [G:H]$.
\end{proof}

\begin{proof}[Theorem \ref{thm: index vs complexity} implies Theorem \ref{thm: isomorphic finite index}]
By Sykiotis \cite{sykiotis2018complexity} it suffices to prove the theorem for one-ended groups.
Let $G$ be a one-ended hyperbolic group.
Following Reznikov \cite{reznikov1996volumes}, we define $$\underline{C}_{2,m}(G) := \liminf_{ H\le G,\; [G:H]<\infty} \frac{C_{2,m}(H)}{[G:H]} $$ where the limit is taken over the lattice of finite index subgroups (note that we do not assume residual finiteness of $G$).
It is easy to see that $\underline{C}_{2,m}(G)$ is multiplicative, that is, for all $H\le G$ of finite index $$\underline{C}_{2,m}(H) = [G:H]\cdot \underline{C}_{2,m}(G).$$
It is also clear that $\underline{C}_{2,m}(G)$ is a group isomorphism invariant.
By Theorem \ref{thm: index vs complexity}, there exists $m\in \bbN$ and $0<\alpha,\beta\in \bbR$ that depend on $G$, such that for every finite index subgroup $H\le G$, $$\alpha\le \frac{C_{2,m}(H)}{[G:H]} \le \beta.$$ Taking the limit inferior we get $0< \underline{C}_{2,m}(G)<\infty.$
It follows that $G$ is finite index rigid: if $H_1,H_2 \le G$ are two isomorphic finite index subgroups of $G$ then $$[G:H_2] \cdot \underline{C}_{2,m}(G) = \underline{C}_{2,m}(H_2)= \underline{C}_{2,m}(H_1) = [G:H_1]\cdot\underline{C}_{2,m}(G)$$
and it follows that $[G:H_1]=[G:H_2]$.
\end{proof}


\paragraph{Acknowledgements.} The author would like to thank Zlil Sela and Alessandro Sisto for suggesting to look at Mineyev's bicombing, Gilbert Levitt for suggesting \cref{cor: free by cyclic},  Michah Sageev for his suggestions regarding the introduction, Alex Margolis for his help with group cohomology, Emily Stark for pointing out mistakes in a previous version, Mladen Bestvina for his help with \cref{lem: Cech cohomology}, and the anonymous referees for their many valuable corrections and improvements.

\section{Notation and conventions}
\label{sec: notation}
\paragraph{Asymptotic notation} We write $M\ll_X N$ if there exists $0<\alpha<\infty$ that depends on $X$ (but does not depend on $M,N$) such that $M\le \alpha N$. 

We write $M\asymp_X N$ if $M\ll_X N$ and $N\ll_X M$, or equivalently there exist $0<\alpha,\beta<\infty$ that depend on $X$ (but not on $M,N$) such that $\alpha N \le M \le \beta N$. Similarly, we write $M\asymp_XN$ when $\alpha,\beta$ depend on $X$.

\paragraph{Constants.}
Throughout the text, a cocompact one-ended simplicial hyperbolic graph $X$ will be fixed.
\textbf{We will use the letter $\adelta$ (in bold) to denote some non-specified ``big'' positive constant that depends only on $X$}, one can think of $\adelta$ as a placeholder for $O_X(1)$ (or more precisely $\Theta_X(1)$ since we assume that the constant is positive). Each appearance of $\adelta$ thus stands for a (possibly) different constant that depends only on $X$. 
For instance, the inequality ``$M \le \adelta N$'' stands for ``$M\ll_X N$'', namely, ``there exists $\delta_1 = \delta_1(X)$ such that $M \le \delta_1 N$''.
In particular, the reader should expect to see things like $2\adelta=\adelta=\adelta^2$,  implications of the sort ``$M_1\le \adelta N_1\text{ and }M_2\le \adelta N_2 \implies M_1 M_2 \le \adelta  N_1 N_2$'' and ``if $r \le \adelta$ then the number of vertices in a ball of radius $r$ in $X$ is at most $\adelta$'', \textit{etc.} 

In some arguments we use $\delta_1,\delta_2,\dots$ (and other Greek letters) to denote specific constants that depend only on $X$. Note that these are not placeholders for $O_X(1)$, but rather explicit numbers, and so we will write explicitly ``there exists $\delta_1 = \delta_1(X)$ such that \dots''.

\paragraph{Simplicial complexes.}
Let $X$ be a simplicial complex. 
By abuse of notation we denote its topological realization by $X$. We denote by $X^i$ its (oriented) $i$-simplices, by $X^{(i)}$ its $i$-skeleton and by $C_i(X)$ its rational $i$-chains.
For a simplex $s\in X^i$ we denote its inverse orientation by $-s$ and its unoriented simplex by $|s|$.
Formally, $s\in X^i$ is an $(i+1)$-tuple of vertices of $X$ ordered up to an even permutation, and $|s|$ is simply the unordered set corresponding to $s$.
Let $a= \sum _s a_s s\in C_i(X)$ be a rational i-chain (where $a_s=0$ for all but finitely many $s\in X^i$). When $i\ge 1$, we will assume that $a_s\ge 0$ and if $a_s > 0$ then $a_{-s}=0$. We denote the oriented support of $a$ by $\supp(a) = \{s \;| \;a_s \ne 0 \}\subseteq X^i$, its unoriented support by $|\supp(a)| = \{|s|\; \mid \; a_s \ne 0\}$, its 1-norm by $\|a\|_1 = \sum_s |a_s|$, and its $\infty$-norm by $\|a\|_\infty = \max_s |a_s|$. On $0$-cycles we also have the augmentation map $\varepsilon (\sum _s a_s s) = \sum_s a_s$.
For an (oriented) edge $e$, denote its endpoints by $e_+,e_-$ so that $\partial e = e_+ - e_-$.

For $x,y\in X^0$, we denote by $d(x,y) \in \nnN$ the length of the shortest path between them in $X^{(1)}$, this length is achieved by some geodesic, by which we mean a sequence of vertices $x=x_0,x_1,\dots,x_n=y$ such that any consecutive pair is an edge of $X$. We denote such a geodesic by $\geod{x}{y}$.
We denote by $[x,y]=\{z\in X^0 \;| \;d(x,z)+d(z,y)=d(x,y)\}$ the interval between $x,y$. Equivalently, $[x,y]$ is the union of all geodesics from $x$ to $y$. 

For a subset $A\subseteq X^0$ we denote its $r$-neighborhood by $N_r(A) = \{ b\in X^0 \;|\; \exists a\in A: d(a,b)\le r\}$, and if the subset is a singleton $A=\{a\}$ we simply denote $B_r(a)$ (the closed ball of radius $r$ around $a$).
For subsets $A,B \subseteq X^0$ we denote by $d(A,B)$ their Hausdorff distance, i.e. $d(A,B)=\inf \{ r \;|\;(B\subseteq N_r(A))\wedge( A\subseteq N_r(B))\}.$
In particular, viewing simplices as subsets of $X^0$ this defines a distance between simplices. 
For a subset $A\subseteq X$, we denote by $\calN_r(A) = \{ f \in X^1 \;|\; \exists a\in A: d(\{a\},f)\le r\}$ the set of edges of $X$ which are at Hausdorff distance at most $r$ from some element of $A$. In other words, the set $\calN_r(A)$ is the set of edges both of whose endpoints are at distance at most $r$ from some element of $A$. As before, $\calB_r(a)=\calN_r(\{a\})$ denotes the edges in the closed ball of radius $r$ around $a$.

A \emph{simplicial action} of a group $G$ on $X$ is an action of $G$ on the vertices of $X$ which maps simplices to simplices. 
Such an action gives rise to an action of $G$ on the realization of $X$ by extending the action on the vertices linearly to simplices.
The action is free if the stabilizers of simplices are trivial in $G$, or in other words, if the induced action on the realization of $X$ is free. 
Note that even if the action $G\actson X$ is free, the quotient $X/G$ is not necessarily simplicial. 
To remedy this we can consider the action of $G$ on the barycentric subdivision $\dot{X}$ of $X$. This action is free and $\dot{X}/G$ is simplicial.
Therefore, throughout this text, by a \emph{free simplicial action} $G\actson X$ we will mean a free simplicial action such that $X/G$ is simplicial. 

Throughout the text, we will assume that $X$ is hyperbolic, \ie for all $x,y,z\in X$ and choices of geodesics $\geod{x}{y},\geod{x}{z},\geod{z}{y}$ connecting them,  $\geod{x}{y}\subset N_\adelta (\geod{x}{z}\cup \geod{z}{y})$. In particular, $[x,y]\subseteq N_\adelta (\geod{x}{y})$.

\section{Globally stable bicombing}\label{sec: global stability}
\begin{definition}
Let $X$ be a simplicial graph, let $G\actson X$ be a simplicial action,  and let $q:X^0\times X^0 \to C_1(X)$. 
\begin{enumerate}[label = (GSB\arabic*)]
    \item \label{gsb: bicombing}$q$ is \emph{a $\bbQ$-bicombing} if for all $x,y\in X^0$, $\partial q(x,y)=y-x$ and $q(x,x)=0$,
    \item \label{gsb: quasigeodesic}$q$ is \emph{quasigeodesic} if for all $x,y\in X^0$ and geodesic $\geod{x}{y}$, $|\supp q(x,y)| \subseteq \calN_\adelta (\geod{x}{y})$ and $\|q(x,y)\|_1\le \adelta \cdot d(x,y)$,
    \item \label{gsb: equivariance}$q$ is \emph{$G$-equivariant} if for all $g\in G$, and $x,y\in X^0$, $gq(x,y)=q(gx,gy)$,
    \item \label{gsb: antisymmetry}$q$ is \emph{anti-symmetric} if for all $x,y\in X^0$, $q(x,y)=-q(y,x)$, and
    \item \label{gsb: defect}$q$ has \emph{bounded defect} if for all $x,y,z\in X^0$, $\|q(x,y)+q(y,z)+q(z,x)\|_1\le \adelta$.
\end{enumerate}
A \emph{globally stable bicombing} is a function $q$ that satisfies all of the above with respect to the group $\Aut(X)$.
\end{definition}

\begin{remark}
When $X$ is cocompact, the inequality $\|q(x,y)\|_1\le \adelta \cdot d(x,y)$ is also implied by the bounded defect and equivariance properties. To see this, note that if $d(x,y)=n$ there is a path $x=x_0,x_1,\dots,x_n=y$ in $X$. By using the bounded defect $n$ times we get $\|q(x,y)\|_1\le \sum_{i=0}^{n-1} \|q(x_i,x_{i+1})\|_1 +\adelta \cdot n $. There are finitely many orbits of edges in $X$, and so by the equivariance of $q$,  $\|q(x_i,x_{i+1})\|_1\le \adelta$. It follows that $$\|q(x,y)\|_1\le \sum_{i=0}^{n-1} \|q(x_i,x_{i+1})\|_1 +\adelta \cdot n \le \adelta \cdot n = \adelta \cdot d(x,y).$$
\end{remark}

\begin{theorem}[{Mineyev \cite[Theorem 10]{mineyev2001straightening}}]
Every locally finite hyperbolic graph with a cocompact group action, supports a globally stable bicombing.
\end{theorem}

\begin{remark}
In \cite{mineyev2001straightening}, the group is assumed to act freely on the graph, however, this is not actually used in the proof except at the very beginning where an equivariant geodesic bicombing is assumed. This can be replaced with the rational geodesic bi-combing $p(x,y)=\frac{1}{n}\sum_{i=1}^n a_i$ where $n=n(x,y)$ is the number of geodesics connecting $x,y$ and $a_1,\dots,a_n$ are the obvious 1-chains corresponding to them. This geodesic bi-combing $p$ is clearly $\Aut(X)$-invariant. 
The rest of the construction in \cite{mineyev2001straightening} is purely geometric and remains $\Aut(X)$-equivariant.
\end{remark}

Throughout the text, assume that a globally stable bicombing $q$ is fixed on our fixed locally finite cocompact hyperbolic graph $X$.

\begin{lemma}\label{boundedness of the bicombing}
For all $x,y\in X^0$, $\|q(x,y)\|_\infty\le \adelta$.
\end{lemma}
\begin{proof}
    Let $e$ be an edge in $\supp q(x,y)$, and let $z=e_+$ be its endpoint.
    By \ref{gsb: quasigeodesic}, there exists $\delta_1=\delta_1(X)$ such that $\supp q(w_1,w_2) \subseteq \calN_{\delta_1}(\geod{w_1}{w_2})$ for all $w_1,w_2\in X$.
    Let $x'$ (resp. $y'$) be the point in $\geod{x}{z}$ (resp. $\geod{y}{z}$) at distance $\delta_1+1$ from $z$ if such exists, otherwise take $x'=x$ (resp. $y'=y$).
    Note that $d(x',y')\le 2\delta_1+2\le \adelta$, and by \ref{gsb: quasigeodesic}, $\|q(x',y')\|_1\le \adelta \cdot d(x',y') \le \adelta$. 
    
    By the choice of $x',y'$ we have $e\notin \calN_{\delta_1}(\geod{x}{x'}) \cup \calN_{\delta_1}(\geod{y'}{y})$,
    and by \ref{gsb: quasigeodesic}, this implies that $e$ is not in the support of $q(x,x')$ nor $q(y,y')$.
    Applying the bounded defect \ref{gsb: defect} twice we get $\|q(x,y) - q(x,x') - q(x',y') -q(y',y)\|_1 \le \adelta $. Since $e$ is not in the support of $q(x,x'),q(y',y)$ we see that the coefficient of $e$ in $q(x,y)$ is bounded by $\|q(x',y')\|_1 + \adelta \le \adelta$. 
\end{proof}

\begin{lemma}\label{lem: lower bound on bicombing}
There exists $\rho$ such that for all $x\ne y\in X^0$ and for all $z\in [x,y]$, the restriction $a_z = q(x,y)|_{\calN_{\rho}(z)}$ of $q(x,y)$ to the edges in the ball of radius $\rho$ around $z$ satisfies $\|a_z\|_1\ge 1$.
\end{lemma}

\begin{proof}
By \ref{gsb: quasigeodesic}, there exists $\delta_1=\delta_1(X)$ such that $q(x',y')\subseteq \calN_{\delta_1}(\geod{x'}{y'})$ for all $x',y'\in X$ and geodesics $\geod{x'}{y'}$. 
Set $\rho=2\delta_1+1$.
Denote $a = q(x,y)$, and let $a_z = q(x,y)|_{\calN_{\rho}(z)}$ as in the statement.
Fix a geodesic $\geod{x}{z}$ between $x$ and $z$, and $\geod{z}{y}$ between $z$ and $y$. Their concatenation $\geod{x}{y}=\geod{x}{z}\geod{z}{y}$ is a geodesic from $x$ to $y$.
Set $a_x$ to be the restriction of $q$ to the set $$\{ e\in \supp a - \supp a_z : \exists w \in \geod{x}{z} : d(w,e) \le \delta_1\}$$
And similarly, $a_y$ is the restriction to 
$$\{ e\in \supp a - \supp a_z : \exists w \in \geod{z}{y} : d(w,e) \le \delta_1\}$$
Since $\supp a \subset \calN_{\delta_1}(\geod{x}{y})$, it is evident that $a = a_x + a_z + a_y$. Thus,  $y-x = \partial a = \partial a_x + \partial a_z + \partial a_y.$
The supports of $\partial a_x$ and $\partial a_y$ are disjoint as otherwise there would be $e$ and $e'$ in the supports of $a_x$ and $a_y$ respectively with a common endpoint, however, if this holds it is easy to see that $e$ or $e'$ belongs to $\calN_{\rho}(z)$ in contradiction to the assumption that $\supp a_z$ is disjoint from both $\supp a_x$ and $\supp a_y$.

If we denote, $b_x = \partial a_x + x$, and $b_y = y - \partial a_y$, then $\partial \alpha_z = b_y-b_x$. 
Since $b_x$ and $x$ are homologous 0-chains they have the same image under the augmentation map $\varepsilon (b_x)= \varepsilon(x) = 1 $. Similarly $\varepsilon(b_y) = \varepsilon (y) =1$.
Finally, since $b_x$ and $b_y$ have disjoint supports, we get 
\begin{equation*}
    \|a_z\|_1 \ge \tfrac12\| \partial a_z\|_1 = \tfrac12\|  b_y - b_x\|_1 = \tfrac12(\|b_y\|_1 + \|b_x\|_1 )\ge \tfrac12 (\varepsilon(b_y) +  \varepsilon(b_x))= 1.
\end{equation*}
\end{proof}

\begin{corollary}\label{lower bound on contribution of edge}
There exists $\lambda=\lambda(X),\rho=\rho(X)$ such that for all $x\ne y\in X^0$ and for all $z\in [x,y]$, there exists an edge in ${B_{\rho}(z)}$ whose coefficient in $q(x,y)$ is at least $1/\lambda$.
\end{corollary}

\begin{proof}
The corollary follows from \cref{lem: lower bound on bicombing} by taking $\rho$ to be as in \cref{lem: lower bound on bicombing} and $\lambda$ to be the maximal number of edges in a ball of radius $\rho$ in $X$.
\end{proof}

\section{Weighted singular patterns}\label{sec: singular patterns}
Dunwoody \cite{dunwoody1985accessibility} introduced patterns on 2-dimensional complexes in order to study actions on trees. 
Inspired by Delzant's foliation \cite{delzant1995image} for the Rips-Sela stable cylinders \cite{rips1995canonical}, we introduce in this section a looser notion of patterns -- \emph{weighted singular pattern} -- that will be used in the setting of actions on hyperbolic spaces. 
Weighted singular patterns will differ from Dunwoody's patterns in three ways: the tracks of the pattern can intersect (similar to the patterns that give rise to CAT(0) cube complexes \cite{beeker2016resolutions,beeker2017cubical}), they can have singular points (similar to Delzant's foliation \cite{delzant1995image}), and they will be weighted. We will see in \S\ref{sec: resolution} how, starting from an action of a group $G$ on a hyperbolic space with a globally stable bicombing, one can define a weighted singular pattern on a 2-dimensional simplicial complex $\bar K$ with $\pi_1(\bar K) = G$. 

\begin{definition}\label{def: pattern}
Let $K$ be a 2-dimensional simplicial complex. 
A \emph{(singular) pattern} $\calF$ on $K$ is an immersed graph $\calF \looparrowright K$ such that:
\begin{enumerate}[label = (P\arabic*)]
    \item The vertices of $\calF$ are called \emph{connectors} and they are sent injectively to the interiors of edges and 2-simplices of $K$, they are called \emph{regular} or \emph{singular} accordingly.
    \item The edges of $\calF$ are called \emph{segments}, and their interiors are sent to straight line segments in the interior of 2-simplices of $K$ (in particular, they intersect transversely), segments which connect two regular connectors are called \emph{regular}, and are otherwise \emph{singular}.
    Every singular segment has a regular connector as one of its endpoints, and it does not meet any other segment in the interior of the 2-simplex to which it belongs.
    \item There are finitely many connectors and segments in each simplex.
    \item Each connector $c$ is the endpoint of exactly one segment in each 2-simplex containing $c$ in its closure.
\end{enumerate}
The connected components of $\calF$ are called \emph{tracks}. 
A \emph{regular pattern} is a pattern without singular connectors. 
\end{definition}

See \cref{fig: singular pattern} for an example of a singular pattern. Note that tracks are connected components of $\calF$ and not its image, and so, in \cref{fig: singular pattern} the pattern has five tracks.

\begin{figure}
    \centering
    \includegraphics{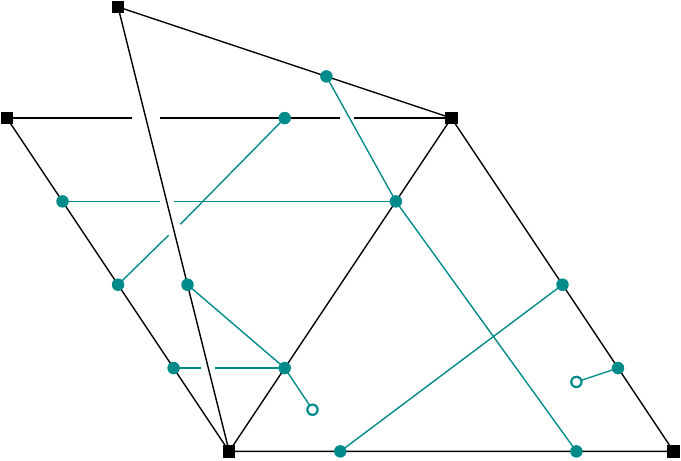}
    \caption{A singular pattern (in cyan) on a 2-complex (in black) consisting of three 2-simplices with a common edge. The regular connectors are shown by full circles, and singular connectors by empty circles.}
    \label{fig: singular pattern}
\end{figure}

\begin{definition}\label{def: weighted pattern}
A \emph{weighted singular pattern}, is a a pair $(\calF,\bfw)$ where $\calF$ is a singular pattern, and $\bfw$ is a function from the set of connectors to non-negative reals $\bbR_{\ge 0}$ 
such that $\bfw (\bfc)=0$ for every singular connector $\bfc$ in $\calF$.

Assuming $K$ is finite, the \emph{weight of a track} $\bft$ is the maximal weight of a connector $\bfw(\bft)=\max_{\bfc\in \bft} \bfw(\bfc)$. The \emph{total weight} of $\calF$ is the sum of weights of its tracks, i.e. $\bfw(\calF)=\sum_{ \bft \subseteq \calF} \bfw( \bft )$.

The \emph{defect} $\df_\bfw(\bfs)$ of a segment $\bfs$ in a weighted singular pattern $(\calF,\bfw)$, is $\df_\bfw(\bfs)=|\bfw(\bfc_1)-\bfw(\bfc_2)|$ where $\bfc_1,\bfc_2$ are the endpoints of $\bfs$.
The \emph{defect} $\df_\bfw( \bft )$ of a track $ \bft $ is defined to be the sum of defects of segments in $ \bft $, i.e. $\df_\bfw( \bft )=\sum_{\bfs\subseteq  \bft } \df_\bfw(\bfs)$.
Finally, the \emph{defect} of $(\calF,\bfw)$ is the sum of defects of its tracks, $\df_\bfw(\calF)=\sum_{ \bft \subseteq \calF} \df_\bfw( \bft )$, or equivalently, the sum of defects of its segments, $ \df_\bfw(\calF)= \sum_{\bfs \subseteq \calF} \df_\bfw(\bfs)$. We will write $\df(\cdot)$ instead of $\df_\bfw(\cdot)$ if $\bfw$ is clear from the context.

A weighted singular pattern $(\calF,\bfw)$ is \emph{perfect} if it is regular and $\df(\calF)=0$.

We denote $\calF'\subset \calF$ if  $\calF'$ is a subgraph of $\calF$, and its immersion $\calF'\looparrowright K$ is the restriction of $\calF\looparrowright K$ to $\calF'$. If $(\calF',\bfw'),(\calF,\bfw)$ are weighted, and $\calF'\subseteq \calF$ we write $\bfw'\le \bfw$ if $\bfw'(\bfc)\le \bfw(\bfc)$ for every connector $\bfc$ in $\calF'$.
\end{definition}

\begin{lemma}\label{lem: removing singularities}
Let $(\calF,\bfw)$ be a weighted singular pattern on a compact simplicial complex $K$. Then there is a perfect weighted pattern $(\calF',\bfw')$ such that $\calF'\subseteq \calF$, $\bfw'\le \bfw$  and $\bfw(\calF)\le \bfw'(\calF')+\df_\bfw(\calF)$.
\end{lemma}

\begin{proof}
Define $\bfw'$ on $\calF$ by $\bfw'(\bfc)=\max\{0,\bfw( \bft )-\df_\bfw( \bft )\}$ for every connector $\bfc$ in a track $ \bft $ of $\calF$. 

Let us first show that $\bfw'\le \bfw$: Let $\bfc$ be a connector, let $ \bft $ be the track containing $\bfc$ and let $\bfc'$ be the connector in $\bft$ with the maximal weight, i.e. $\bfw( \bft )=\bfw(\bfc')$. 
There is a (simple) path of segments $\bfs_1,\dots,\bfs_n$ in $ \bft $ between $\bfc$ and $\bfc'$. By the definition of the defect, we get 
\[\bfw( \bft )-\bfw(\bfc)= \bfw(\bfc')-\bfw(\bfc) \le \df_\bfw(\bfs_1)+\dots +\df_\bfw(\bfs_n) \le \df_\bfw( \bft ).\] It follows that $\bfw'(\bfc)=\max\{0,\bfw( \bft )-\df_\bfw( \bft )\}\le \bfw(\bfc)$.

In particular, if $\bfc$ is a singular connector then $\bfw'(\bfc)\le \bfw(\bfc)=0$. Hence, $(\calF,\bfw')$ is a weighted singular pattern.

By definition, all connectors in $ \bft $ have the same weight with respect to $\bfw'$, and so $\df_{\bfw'}(\calF)=0$.  Define $\calF'$ to be the tracks in $\calF$ on which $\bfw'\ne 0$, we immediately get that $\calF'$ is regular (since any singular connector has weight $0$).
Finally, the equality $\bfw(\calF) \le \bfw'(\calF')+\df_\bfw(\calF)$ is clear from the definition of $\bfw'$.
\end{proof}

\section{Resolutions of globally stable bicombings}\label{sec: resolution}

Let $G$ be a group, and let $K$ be a simply-connected 2-dimensional simplicial complex with a free, cocompact, simplicial action $G\actson K$.
Let $G\actson X$ be a free and cocompact simplicial action on a $\adelta$-hyperbolic graph $X$ with a globally stable bicombing $q$. Let $\Phi: K^0\to X^0$ be a $G$-equivariant map.

For an edge $e\in K^1$, denote $q(e)=q(\Phi(e_-),\Phi(e_+))$. Observe that by \ref{gsb: antisymmetry} and \ref{gsb: equivariance} this definition is anti-symmetric, $q(-e)=-q(e)$, and $G$-equivariant, $q(g. e) = g. q( e)$.

\begin{definition}
A \emph{resolution of $q$ to $K$ via $\Phi$} is a $G$-invariant weighted singular pattern $(\calF,\bfw)$ on $K$ together with a bijection for every edge $e\in K^1$ $$\bfc(e,\cdot):\supp q(e) \to \{\text{regular connectors of $\calF$ on }e\}$$ such that the following hold for every edge $e\in K^1$:
\begin{enumerate}[label=(R\arabic*)]
    \item (symmetry)\label{condition: symmetry} $ \bfc(- e,-f)= \bfc( e,f)$  for all $f\in \supp q( e)$;
    \item ($G$-equivariance)\label{condition: G-equivariance} $g . \bfc( e, f) =  \bfc(g.  e, g.f)$ for all  $g\in G$ and $f\in \supp q( e)$;
    \item (quasi-ordered) \label{condition: quasi-ordered}  For all $f,f'\in \supp q( e)$ $$d(\Phi(e_-), f) < d(\Phi(e_-),f')-\adelta \implies \bfc( e, f)<\bfc( e,f')\text{ on $e$}$$ where we think of $e$ as an interval ordered from $e_-$ to $e_+$ (see \cref{fig:regular connectors} for an illustration);
    \item \label{condition: weight is coefficient} For all $f\in \supp q(e)$ the weight $\bfw(c(e,f))$ is the coefficient of $f$ in $q(e)$;
\end{enumerate}
    and for every 2-simplex $\Delta\in K^2$ the following holds:
    \begin{enumerate}[label=(R\arabic*)]\setcounter{enumi}{4}
    \item \label{condition: regular segments} two regular connectors $\bfc,\bfc'$ on the boundary of $|\Delta|$ are connected by a regular segment if and only if we can write $\bfc=\bfc(e,f),\bfc'=\bfc(e',-f)$ such that $e,e',e''$ are the edges of $\partial\Delta$, oriented such that $\partial\Delta=e+e'+e''$,  the edge $f$ of $X$ belongs to $\supp q(e)$, $-f$ belongs to $\supp q(e')$, and neither of $\pm f $ belongs to $ \supp q(e'')$. See \cref{example illustration} and the accompanying \cref{fig:regular and singular segments}.
\end{enumerate}
\end{definition}

\begin{figure}
    \centering
    \includegraphics{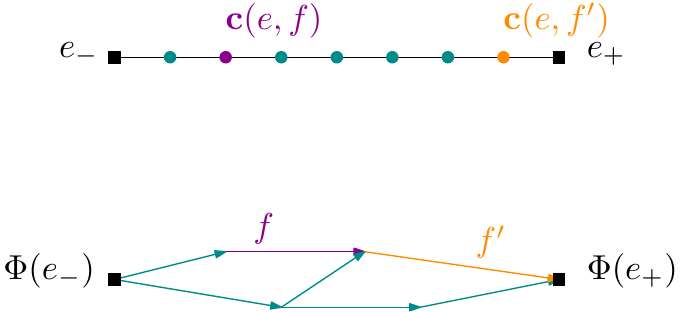}
    \caption{The regular connectors on the edge $e$, on top, correspond to edges in $\supp q(e)$, at the bottom. \ref{condition: quasi-ordered} says that if $f$ is (significantly) closer than $f'$ to $\Phi(e_-)$ then $\bfc(f,e)$ is closer than $\bfc(f',e)$ to $e_-$.}
    \label{fig:regular connectors}
\end{figure}

\begin{example}\label{example illustration}
To illustrate \ref{condition: regular segments} consider the 2-simplex $\Delta$, with boundary edges oriented so that $e+e'+e''=\partial \Delta$. Consider four edges $f_1,f_2,f_3,f_4$ in $X$ and let the supports of $q(e),q(e'),q(e'')$ be such that
\begin{align*}
        \supp q(e) \cap \{\pm f_1,\dots,\pm f_4\} &= \{f_1,f_2,f_3\}\\
        \supp q(e') \cap \{\pm f_1,\dots,\pm f_4\} &= \{-f_1,f_2\}\\
        \supp q(e'') \cap \{\pm f_1,\dots,\pm f_4\} &= \{-f_2,f_3,f_4\}
\end{align*}
As depicted also in \cref{fig:regular and singular segments}.
The edge $f_1$ belongs to $\supp q(e)$, $-f$ belongs to $\supp q(e')$ and the edges $\pm f_1$ do not belong to $\supp q(e'')$, hence $\bfc(e,f_1)$ and $\bfc(e',-f_1)$ are joined by a regular segment. The connectors corresponding to the edges $f_2,f_3,f_4$ do not satisfy this condition: each of the three supports contains either $f_2$ or $-f_2$; $f_3$ belongs to both $\supp q(e)$ and $\supp q(e'')$; $f_4$ (in either orientation) belongs to only one of the supports, namely $\supp q(e'')$. Hence the connectors corresponding to these edges are connected by a short singular segment to a singular connector.
\begin{figure}
    \centering
    \includegraphics{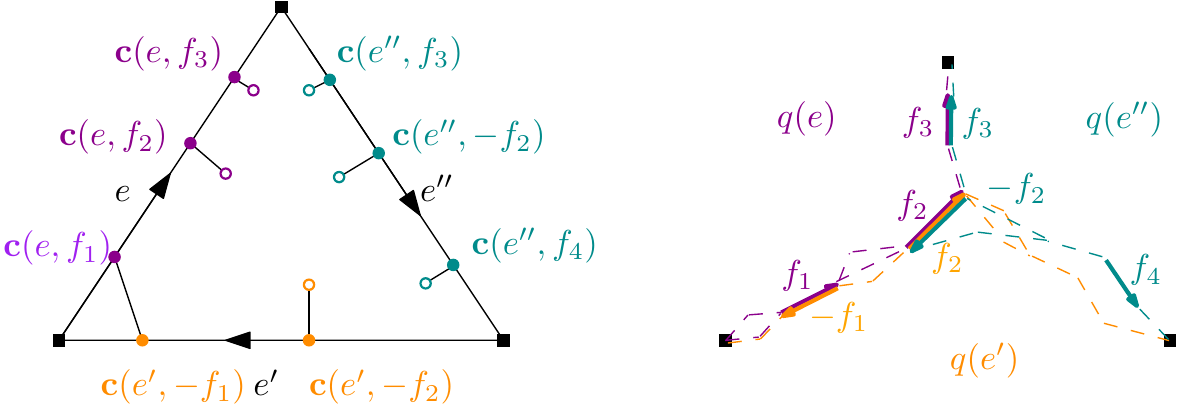}
    \caption{On the left, we see the connectors and the segments corresponding to the edges $|f_1|,|f_2|,|f_3|,|f_4|$ which appear in (some orientation in some of) the supports of $q(e), q(e'), q(e'')$ which are depicted on the right in three colors. 
    }
    \label{fig:regular and singular segments}
\end{figure}
\end{example}

\begin{lemma}
For every $G,K,\Phi,X,q$ as above, there exists a resolution $(\calF,\bfw)$ of $q$ to $K$ via $\Phi$.
\end{lemma}
\begin{proof}
    Let $v_1,\dots,v_l$, $\pm e_1,\dots,\pm e_m$, and $\pm \Delta_1,\dots,\pm \Delta_n$ be representatives for the $G$-orbits of vertices, edges and 2-simplices respectively.
    We begin by constructing the regular connectors of $\calF$ and the maps $ \bfc( e,\cdot)$. First, choose some injective map  $ \bfc( e_j,\cdot):\supp q( e_j) \to  e_j$ for each $1\le j \le m$ satisfying:
\begin{enumerate}[label=(R\arabic*')]
\setcounter{enumi}{2}
    \item  \label{condition: quasi-ordered'}If $f,f'\in \supp q( e_j)$ are such that $d(\Phi((e_j)_-), f) < d(\Phi((e_j)_-),f')$ then $\bfc(e_j, f)<\bfc(e_j,f')$ on $e_j$.
\end{enumerate}
Clearly \ref{condition: quasi-ordered'} implies \ref{condition: quasi-ordered} for the edge $e=e_j$.
Define $\bfc(- e_j,f) =  \bfc ( e_j,-f)$ for all $f\in \supp q(- e_j)$.

\begin{claim} \ref{condition: quasi-ordered} holds for the edge $e=-e_j$.
\end{claim}
\begin{proof} 
Denote by $u_-=(e_j)_-$ and $u_+=(e_j)_+$ the endpoints of $e_j$. Note that $u_+=(-e_j)_-$ and $u_-=(-e_j)_+$.
Let $f,f'$ be edges in the support of $q(-e_j)$. Assume that $\bfc(-e_j,f)>\bfc(-e_j,f')$ on $-e_j$. We want to show that $d(\Phi(u_+),f) \ge d(\Phi(u_+),f') - \adelta$.

We have $\bfc(e_j,-f)<\bfc(e_j,-f')$ on $e_j$. By \ref{condition: quasi-ordered'}, it follows that $d(\Phi(u_-),f) \le d(\Phi(u_-),f')$.
Let $\gamma = \geod{\Phi(u_-)}{\Phi(u_+)}$. By \ref{gsb: quasigeodesic}, there exist $y,y'\in\gamma$ such that $d(y,f),d(y',f')\le \adelta$. 
By the triangle inequality we get
$$ d(\Phi(u_-),y) \le d(\Phi(u_-),f) + \adelta \le d(\Phi(u_-),f') +\adelta \le d(\Phi(u_-),y') + \adelta$$ 
Since $y,y'$ are on the geodesic $\gamma$ between $\Phi(u_-),\Phi(u_+)$ we get 
$$ d(\Phi(u_+),y) \ge  d(\Phi(u_+),y') - \adelta$$ 
By the triangle inequality, we get
$$ d(\Phi(u_+),f) \ge d(\Phi(u_+),y) - \adelta \ge  d(\Phi(u_+),y') - \adelta \ge d(\Phi(u_+),f) - \adelta,$$
as desired.
\end{proof}

We have defined $ \bfc (\pm  e_j, \cdot)$ to satisfy \ref{condition: symmetry} and \ref{condition: quasi-ordered}. Define $\bfc(e,\cdot)$ for all the edges $e$ of $K$ by extending $G$-equivariantly the definition of $\bfc(\pm e_j,\cdot)$, so as to satisfy \ref{condition: G-equivariance}.

The regular connectors of $\calF$ are the points $$\{ \bfc( e, f)\;|\; e\in  K,\;f\in \supp q( e)\}.$$
Set the weight of the regular connector $ \bfc( e, f)$ to be the coefficient of $f$ in the 1-chain $q(e)$ to satisfy \ref{condition: weight is coefficient}.
In each 2-simplex $\Delta$, connect two regular connectors by a regular segment as instructed by \ref{condition: regular segments}. Note that in each 2-simplex, every regular connector is connected by a regular segment to at most one other regular connector.

It remains to define the singular connectors and segments. If a regular connector $\bfc=\bfc(e,f)$ in the 2-simplex $\Delta_k$ is not connected to another regular connector in $\Delta_k$, add a short segment connecting $\bfc(e,f)$ to a newly added singular connector $ \bfc ( \Delta_k, e,f)$, and set $\bfw(\bfc(\Delta_k,e,f))=0$.
By choosing the singular segments to be short enough, we may assume that they do not intersect any other segment.

Finally, extend this construction $G-$equivariantly to all 2-simplices of $K$. The obtained weighted singular pattern $(\calF,\bfw)$ is a resolution of $q$ to $K$ via $\Phi$.
\end{proof}

\begin{lemma} \label{lem: observations about tracks}
Let $(\calF,\bfw)$ be a resolution of $q$ to $K$ via $\Phi$, then:
\begin{enumerate}[label = (T\arabic*)]
    \item \label{lem: observations about tracks: embedded 2sided}Each track of $ \calF$ is embedded and meets each edge of $ K$ at most once.
    \item \label{lem: observations about tracks: correspondence to edges} Each track corresponds to a unique unoriented edge $|f|$ in $X$.
    \item \label{lem: observations about tracks: compact trivial stab}Each track is compact and its stabilizer is finite.
    \item \label{lem: observations about tracks: intersecting tracks}If two tracks $\bft,\bft'$ in $ \calF$ intersect then the corresponding edges $f,f'$ in $X$ satisfy $d(f,f')<\adelta$.
\end{enumerate}
\end{lemma}

\begin{proof}
\cref{lem: observations about tracks: embedded 2sided,lem: observations about tracks: correspondence to edges} are clear from the definition.
To prove \cref{lem: observations about tracks: compact trivial stab} note that the 
map that sends a track to its corresponding unoriented edge in $X$ is $G$-equivariant. Since the action on $X$ is free, it follows that the stabilizer of each track is trivial. 
To prove that each track is compact, we first observe that since the action of $G$ on $K$ is cocompact, the action of $G$ on the set of connectors in $\calF$ has finitely many orbits.
The map that sends each connector of $\calF$ to its corresponding edge in $X$ is a $G$-equivariant map from a set with finitely many $G$-orbits to a set with a free $G$-action. It follows that this map is finite-to-one. Hence, each track has finitely many connectors, and so must be compact.

Finally, to prove \cref{lem: observations about tracks: intersecting tracks}, let $\bft,\bft'$ be two tracks and let $|f|,|f'|$ be their corresponding edges in $X$. Note that if two tracks intersect then both tracks meet a 2-simplex $\Delta$ of $ K$ in two intersecting regular segments $\bfs,\bfs'$. Let $e$ be an edge which meets both $\bft,\bft'$. Thus, up to reversing their orientation if necessary, we may assume that $f,f'\in\supp q( e)$. So $\bft,\bft'$ meet the edge $e$ in the connectors $ \bfc ( e,f), \bfc ( e, f')$ which without loss of generality we may assume that $ \bfc ( e,f)< \bfc ( e, f')$ on $ e$.
One of two things can happen:

\begin{figure}
    \centering
    \begin{subfigure}[b]{0.4\textwidth}
         \centering
         \includegraphics[width=\textwidth]{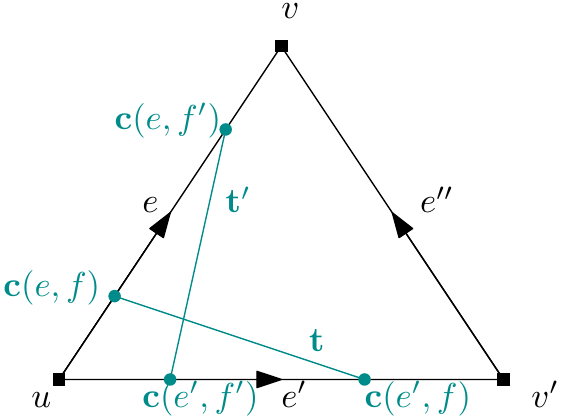}
         \caption{Case 1}
         \label{fig:crossing tracks 1}
     \end{subfigure}
     \hfill
     \begin{subfigure}[b]{0.4\textwidth}
         \centering
         \includegraphics[width=\textwidth]{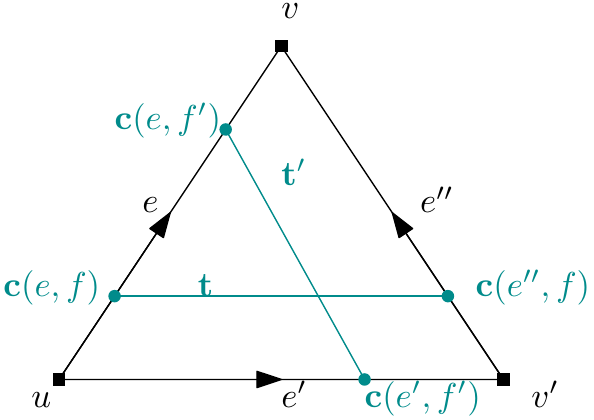}
         \caption{Case 2}
         \label{fig:crossing tracks 2}
     \end{subfigure}
    \caption{Crossing tracks}
    \label{fig:crossing tracks}
\end{figure}

\textbf{Case 1.} The other endpoints of the segments $\bfs,\bfs'$ are on the same edge, say $ e'$, as in \cref{fig:crossing tracks 1}. 
Let us assume that $ e'$ is oriented so that $\partial  \Delta =  e -  e' -  e''$ and $e_-=e'_- = u$. The other connectors of $\bfs,\bfs'$ are $ \bfc (e',f),  \bfc ( e',f')$.
Since the segments meet we must have $ \bfc (e',f) >  \bfc (e',f')$ on $ e'$. By \ref{condition: quasi-ordered} we have $| d(\Phi(u), f) - d(\Phi(u),f') | < \adelta$. Moreover, $f,f'\in q( e)$, so since $q$ is quasigeodesic \ref{gsb: quasigeodesic}, then they are $\adelta$-close to some points on a geodesic $\geod{\Phi(e_-)}{\Phi(e_+)}$ and by the triangle inequality it is easy to see that $d(f,f')\le \adelta$.

\textbf{Case 2.} The other endpoints of the segments  $\bfs,\bfs'$ are on different edges, as in \cref{fig:crossing tracks 2}. Assume again that $\partial  \Delta =  e -  e' -  e''$, and that $ e_-=e'_- = u$, and $e_+=v$,$e'_+=v'$. 
Let $\gamma = \geod{\Phi(u)}{\Phi(v)}$, $\gamma' = \geod{\Phi(u)}{\Phi(v')}$ and $\gamma'' = \geod{\Phi(v)}{\Phi(v')}$.
Since $ \bfc ( e,f)< \bfc ( e, f')$ on $ e$, we have $d(\Phi(u), f)\le d(\Phi (u),f') + \adelta$.
By \ref{gsb: quasigeodesic}, there are points $x_1,x_2\in\gamma$ which are at distance at most $\adelta$ from $f,f'$ respectively.
We thus have $$d(\Phi(u),x_1) - d(\Phi(u),x_2) \le \adelta.$$
On the other hand, since the segments $\bfs,\bfs'$ cross, the other connector of $\bfs'$ is $ \bfc( e',f')$ on $ e'$ and the connector of $\bfs$ is $ \bfc (  e'' , f)$ on $ e''$. 
By \ref{gsb: quasigeodesic}, there are points $x_1''\in\gamma''$ and $x_2'\in \gamma'$ which are at distance at most $\adelta$ from $f$ and $f'$ respectively.
If we denote the Gromov product $D=(\Phi(v),\Phi(v'))_{\Phi(u)}$ then since the points $d(x_1,x_1'')\le \adelta$ we must have $d(\Phi(u),x_1) \ge D - \adelta$, and similarly, $d(\Phi(u),x_2) \le D + \adelta$. In particular,
$$d(\Phi(u),x_2) - d(\Phi(u),x_1) \le \adelta.$$ 
All together we get $|d(\Phi(u),x_1) - d(\Phi(u),x_2)|\le \adelta$. It follows that $d(x_1,x_2)\le \adelta$ and finally by the triangle inequality that $d(f,f')\le \adelta$.
\end{proof}

Since a resolution $(\calF ,\bfw)$ of $q$ to $K$ via $\Phi$ is $G$-invariant, it defines a weighted pattern $(\bar \calF,\bfw)$ on $\bar K$, which we will call \emph{a resolution of $q$ to $\bar K$ via $\Phi$.}


\section{Bounding the total weight of a resolution}\label{sec: accessibility}

\begin{proposition}\label{accessibility}
Let $X$ be a one-ended hyperbolic graph with a globally stable bicombing $q$. For every proper and cocompact action $G\actson X$, and every simply connected 2-dimensional simplicial complex $K$ with a free cocompact simplicial action of $G \actson K$,
there exists a resolution $\bar\calF$ of $q$ to $\bar K = K/G$ that satisfies
\begin{equation}\label{eq: accessibility inequality}
\bfw(\bar \calF)\le \adelta \cdot \Vol(\bar K).
\end{equation}
\end{proposition}

\begin{proof}
As before, let $v_1,\dots,v_l,$ $e_1,\dots,e_m,$ and $\Delta_1,\dots,\Delta_n$ be the $G$-orbit representatives of vertices, edges and 2-simplices of $K$.
By definition, $\Vol(\bar K) = l+m+n$.
Assume that a $G$-equivariant map $\Phi:K^0 \to X^0$ is chosen so that the displacement $$\sum_{i=1}^m d(\Phi(( e_i)_-),\Phi(( e_i)_+))$$ is minimal among all $G$-equivariant maps.

Denote by $( \calF,\bfw)$ and $(\bar \calF,\bfw)$ the resolutions of $q$ to $K$ and $\bar K$ via $\Phi$.
By Lemma \ref{lem: removing singularities}, there exists a perfect weighted pattern $(\bar \calF',\bfw')$ such that $\bar \calF' \subseteq \bar \calF$, $\bfw'\le \bfw$ and 
\begin{equation}\label{eq: singular regular defect inequality}
    \bfw(\bar \calF) \le \df(\bar \calF) + \bfw'(\bar\calF')
\end{equation}
In Claims \ref{claim: inequality for defect} and \ref{claim: inequality for regular} we will bound each of $\df(\bar\calF)$ and $\bfw'(\bar\calF')$.

\begin{claim}\label{claim: inequality for defect}
$\df(\bar\calF) \le \adelta\cdot n$.
\end{claim}

\begin{proof}
Let $\Delta$ be a 2-simplex of $K$ and let $v,v',v''$ be its vertices and $ e,  e', e''$ its edges (and $\partial \Delta = e+e'+e''$).
Consider the edges in $\supp q(e) \cup \supp q(e') \cup \supp q(e'')$.
We partition them into two classes $\calA$ and $\calB$. The set $\calA$ consists of all edges $f$ such that each of the supports $\supp q( e),\supp q( e'),\supp q( e'')$ contains either $f$ or $-f$. For example, the edge $f_2$ in \cref{fig:regular and singular segments} is such. The set $\calB$ is the complement of $\calA$.

By \ref{gsb: quasigeodesic}, an edge $f\in\cal A$ is at distance $\adelta$ from the sides of the geodesic triangle spanned by $\Phi(v),\Phi(v'),\Phi(v'')$, and hence at distance $\adelta$ from the center of the triangle.
It follows that the number of edges in $\calA$ is at most $\adelta$.
To each $f\in\calA$, we associate three singular segments of $\calF$ in $\Delta$. 
By \cref{boundedness of the bicombing} the defect of each of the singular segments is at most $\adelta$.
Therefore, the sum of the defects of all the singular segments corresponding to the edges in $\calA$ is $\adelta$.

As for edges $f\in \calB$, there are three cases to consider (corresponding to the three edges $f_1,f_4,f_3$ in  \cref{example illustration} respectively):

    \textbf{Case 1.} If $f$ is in one of the supports, and $-f$ is in another support. Say, $f\in \supp q(e)$ and $-f\in \supp q(e')$. By \ref{condition: regular segments}, the two connectors $\bfc(e,f),\bfc(e',-f)$ are connected by a regular segment $\bfs$ whose defect $\df(\bfs)=|\bfw(\bfc(e,f))-\bfw(\bfc(e',-f))|$ is equal to the absolute value of the coefficient of $f$ in the sum $q(e)+q(e')+q(e'')$.
    
    \textbf{Case 2.} If $f$ is in exactly one of the three supports, say $f\in \supp q(e)$, and $-f$ is in none of them. Then, there is a singular segment $\bfs$ connecting $\bfc(e,f)$ to a singular connector. The defect $\df(\bfs) = |\bfw(\bfc(e,f))|$ is equal to the absolute value of the coefficient of $f$ in the sum $q(e)+q(e')+q(e'')$.
    
    \textbf{Case 3.} If $f$ is in two out of  the three supports, say $f\in \supp q(e)$ and $f\in \supp q(e')$. In this case, $f$ gives rise to two singular segments $\bfs,\bfs'$ connecting the regular connectors $\bfc(e,f),\bfc(e',f)$ to the singular connectors $\bfc(\Delta,e,f),\bfc(\Delta,e',f)$ respectively. The sum of their defects $\df(\bfs)+\df(\bfs')=|\bfw(\bfc(e,f))|+|\bfw(\bfc(e',f))|$ is equal to the absolute value of the coefficient of $f$ in the sum $q(e)+q(e')+q(e'')$. 

In all three cases the sum of the defects of the segments corresponding to $f\in \calB$ in $\Delta$ is the absolute value of the coefficient of $f$ in the sum $q(e)+q(e')+q(e'')$.
Therefore, the sum of the defects of all segments in $\Delta$ corresponding to edges in $\calB$ is bounded above by $\|q(e)+q(e')+q(e'')\|_1$ which by \ref{gsb: defect} is bounded by $\adelta$.

Thus the sum of the defects of all the segments of $\calF$ in $\Delta$ is bounded by $\adelta$. It follows that the sum of the defects of the segments of $\bar \calF$ in a 2-simplex $\bar \Delta$ of $\bar K$ is bounded by $\adelta$. 
Since there are $n$ 2-simplices in $\bar K$ we get that the total defect $\df(\bar\calF)$, which is the sum of the defects of all the segments in $\bar\calF$, is bounded by $$\df(\bar \calF)=\sum_{\bfs\subseteq \calF} \df(\bfs) \le \adelta \cdot n.$$
\end{proof}

Denote by $\calF'$ the pattern on $ K$ obtained by considering all lifts of the pattern $\bar \calF'$.
Every track $ \bft$ of $ \calF'$ is compact, embedded and locally 2-sided. Since $K$ is simply-connected, each track $\bft$ of $\calF'$ separates $K$ into two components. Since $G$ is one-ended, $K$ is one-ended, and it follows that one of the components is bounded and the other is unbounded. We denote by $\bft_+$ the bounded component and $ \bft_-$ the unbounded component. 

If a vertex $v_i$ is contained in $\bft_+$ for some track $\bft$ of $\calF'$, let $\bft(v_i)$ be an outermost track such that $v_i\in \bft(v_i)_+$, that is, such that there is no track $\bft'$ in $\calF'$ such that $ \bft( v_i)_+ \subset \bft'_+$. 
Otherwise, for ease of notation, set $\bft(v_i)=v_i$.
Extend the map $\bft(\cdot):K^0\to \{\text{tracks in $\calF'$}\}\cup K^0$ to all vertices $v$ of $ K$  $G$-equivariantly.
Set $y_i = \Phi(v_i)$ if $\bft(v_i)=v_i$, and otherwise set $y_i = (f_i)_+$ to be an endpoint of the edge $f_i$ corresponding $ \bft( v_i)$. Let $\Psi:K^0 \to X^0$ be the $G$-equivariant map defined by $\Psi(g.v_i)=g.y_i$. By $G$-equivariance, $\Psi(v)$ is either $\Phi(v)$ or an endpoint of the edge corresponding to $\bft(v)$ depending on whether $\bft(v)=v$ or $\bft$ is a track.

Finally, let $ \calF'_i$ be the collection of tracks meeting the edge $ e_i$ in $\calF'$.

\begin{claim}\label{claim inequality per edge}
For all $1\le i\le m$,
\begin{equation}\label{eq: inequality per edge}
    d(\Psi ((e_i )_-),\Psi ((e_i )_+)) \le d(\Phi ((e_i )_-),\Phi ((e_i )_+)) - \bfw'( \calF'_i)/\adelta + \adelta.
\end{equation} 
\end{claim}

\begin{proof}
Let $1\le i \le m$, denote $ e= e_i$. Let its endpoints be $ v= e_-,  v'= e_+$, and let their images under $\Phi$ be $x,x'$, and under $\Psi$ be $y,y'$ respectively.
Thus we want to show $$d(y,y')\le d(x,x') - \bfw'( \calF'_i)/\adelta + \adelta.$$

Consider $ \bft =  \bft ( v),  \bft' =  \bft( v')$  for the two endpoints $v,v'$ of $ e$. We divide into cases:

\textbf{Case 1.} $\bft=v,\bft'=v'$. 

By the choice of the map $\Psi$ we get in this case $y=x,y'=x'$.  It also follows that no track $\bft$ of $\calF'$ can meet $e$, as otherwise $v$ or $v'$ would be contained in $\bft_+$. Thus, $\bfw'(\calF'_i)=0$, and  the desired inequality follows.

\textbf{Case 2.} $\bft\ne v,\bft'\ne v'$. 

Let $f,f'$ be the edges corresponding to $\bft,\bft'$ respectively. Recall that $y,y'$ are endpoints of $f,f'$.
We divide into two subcases depending on whether $\bft,  \bft'$ intersect.

\textbf{Case 2.1.} $ \bft \cap  \bft' \ne \emptyset$.

Since the tracks meet, by \cref{lem: observations about tracks: intersecting tracks}, their corresponding edges satisfy $d(f,f')\le \adelta$. Thus, $$d(y,y') \le \adelta.$$

 Since $(\calF',\bfw')$ is a perfect weighted pattern, $\bfw'(\bft) = \bfw'(\bfc)$ for any connector $\bfc$ on the track $\bft$ in $\calF'$. Moreover, we know $\bfw'(\bfc) \le \bfw(\bfc)$ for every connector $\bfc$ of $\calF'$. Summing over the connectors $\bfc$ of the singular patterns $\calF',\calF$ on the edge $e_i$ we get
 $$\bfw'(\calF'_i) = \sum_{\bfc \in e_i} \bfw'(\bfc) \le \sum_{\bfc \in  e_i} \bfw(\bfc) = \|q( e)\|_1\le \adelta \cdot d(x,x')$$ where the last inequality follows from \ref{gsb: quasigeodesic}.
Hence,
$$ 0\le d(x,x') - \bfw'(\calF'_i)/\adelta.$$
Combining the two inequalities we get the desired inequality
$$ d(y,y') \le \adelta \le d(x,x') - \bfw'(\calF'_i)/\adelta + \adelta.$$

\textbf{Case 2.2.} $ \bft \cap  \bft' = \emptyset$.

In this case, $ \bft$ and $\bft'$ intersect $ e$ as otherwise both $ \bft_+, \bft'_+$ contain the same endpoint of $ e$, however since they do not intersect then either $ \bft_+ \subseteq  \bft'_+$ or $ \bft'_+ \subseteq  \bft_+$ in contradiction to the assumption that $ \bft, \bft'$ are outermost.
Let $ \bfc, \bfc'$ be the connectors in which $ \bft, \bft'$ meet $ e_i$. Since $ \bft, \bft'$ do not intersect we must have $ \bfc<  \bfc'$ on $ e_i$.

The connectors $\bfc,\bfc'$ separate the edge $e_i$ to three intervals, $I_1 = [ v, \bfc],I_2 = ( \bfc, \bfc')$ and $I_3 = [ \bfc', v']$, and accordingly partition the set $ \calF_i'$ into three subsets $ \calF_{i,1}', \calF_{i,2}', \calF_{i,3}'$ of tracks which meet the intervals $I_1,I_2,I_3$ respectively.

Bounding $\bfw'( \calF_{i,2}')$: If $ \bft''$ has a connector on $I_2$ then $\bft''_+$ contains one of the endpoints of $e_i$. Since $ \bft,  \bft'$ are outermost, $ \bft''$ must intersect $ \bft$ or $ \bft'$.
By \ref{lem: observations about tracks: intersecting tracks} of \cref{lem: observations about tracks}, the edge that corresponds to $ \bft''$ is at distance at most $\adelta$ from one of the edges $f,f'$ corresponding to $\bft,\bft'$.
The number of such edges is bounded by $\adelta$.
This implies that there are at most $\adelta$ tracks that cross $I_2$.
By \cref{boundedness of the bicombing}, the weight of each such track is at most $\adelta$ and so
\begin{equation}\label{eq: bound for Fi2}
    \bfw'( \calF'_{i,2})\le \adelta.
\end{equation}

Bounding $\bfw'( \calF_{i,1}')$ and $\bfw'( \calF_{i,3}')$: Let $\gamma = \geod{x}{x'}$. By \ref{gsb: quasigeodesic}, there are points $w,w'$ on $\gamma$ at distance at most $\adelta$ from $f,f'$ respectively. By \ref{condition: quasi-ordered}, the edge $f''$ corresponding to a track $\bft''$ in $\calF_{i,1}'$ satisfies $d(x,f'')\le d(x,f)+\adelta$. By \ref{gsb: quasigeodesic}, and some application of the triangle inequality, shows that $f''$ is contained in the $\adelta$ neighborhood of the subsegment $\geod{x}{w}\subseteq \gamma$. In particular there are at most $\adelta d(x,w)$ such edges.
By \cref{boundedness of the bicombing} the weight of each such track is at most $\adelta$ and so
\begin{equation}\label{eq: bound for Fi1}
    \bfw'( \calF_{i,1}')\le \adelta d(x,w).
\end{equation}
Similarly,
\begin{equation}\label{eq: bound for Fi3}
    \bfw'( \calF_{i,3}')\le \adelta d(x',w').
\end{equation}

Since $\bfw'(\calF_i') = \bfw'( \calF_{i,1}') + \bfw'(\calF_{i,2}') +  \bfw'( \calF_{i,3}')$ we get
\begin{align*}
    d(y,y') &\le \adelta + d(w,w')\\
    &= \adelta + d(x,x') - d(x,w) - d(x',w')\\
    &\le d(x,x') - (\bfw'( \calF_{i,1}') + \bfw'( \calF_{i,3}'))/\adelta + \adelta \text{\quad\quad By \eqref{eq: bound for Fi1}\eqref{eq: bound for Fi3}}\\
    &\le d(x,x') - (\bfw'( \calF_{i,1}') + \bfw'(\calF_{i,2}') +  \bfw'( \calF_{i,3}'))/\adelta +  \adelta + \bfw'(\calF_{i,2}')/\adelta\\
    &\le d(x,x') - \bfw'(\calF_i')/\adelta + \adelta \text{\quad\quad By \eqref{eq: bound for Fi2}}
\end{align*}
Thus, also in Case 2 we get the desired inequality.

\textbf{Case 3.} $\bft\ne v,\bft'= v'$. This case is done similarly to Case 2.2 above. The track $\bft$ must meet the edge $e$, as otherwise $v'\in \bft_+$. Let $\bfc$ be the connector of $\bft$ on $e$.  Partition $e$ into two intervals $I_1=[v,\bfc]$ and $I_2=(\bfc,v']$. The rest of the argument is almost identical, and so we skip it.
\end{proof}

\begin{claim}\label{claim: inequality for regular}
$\bfw'(\bar\calF') \le \adelta m$ where $m$ is the number of edges of $\bar K$.
\end{claim}

\begin{proof}
By \cref{claim inequality per edge} and by the assumption of minimal displacement of $\Phi$ we get 
\begin{align*}
    \sum_{i=1}^m d(\Phi(( e_i)_-,\Phi(( e_i)_+) & \le \sum_{i=1}^m d(\Psi(( e_i)_-,\Psi(( e_i)_+) \\
    &\le \sum_{i=1}^m \left(d(\Phi(( e_i)_-,\Phi(( e_i)_+) - \bfw'( \calF'_i)/\adelta + \adelta\right)\\
    &\le \sum_{i=1}^m d(\Phi(( e_i)_-,\Phi(( e_i)_+) - \sum _{i=1}^{m} \bfw'( \calF'_i)/\adelta + \adelta\cdot m.
\end{align*} 
Subtracting $\sum_{i=1}^m d(\Phi(( e_i)_-,\Phi(( e_i)_+) - \sum _{i=1}^{m} \bfw'( \calF'_i)/\adelta$ from both sides we get,
$$ \sum _{i=1}^{m} \bfw'( \calF'_i)/\adelta \le  \adelta \cdot m$$
which implies
\[\bfw'(\bar\calF') \le \sum_{i=1}^m\bfw'( \calF'_i) \le \adelta\cdot m.\]
\end{proof}


By combining \cref{eq: singular regular defect inequality} and Claims \ref{claim: inequality for defect} and \ref{claim: inequality for regular} we get the desired inequality \cref{eq: accessibility inequality}
\[\bfw(\bar\calF)\le \df(\bar\calF)+\bfw'(\bar\calF') \le \adelta \cdot m + \adelta \cdot n \le \adelta \cdot \Vol (\bar K).\]
\end{proof}

\begin{remark}
    Since $G$ is one-ended, one can reduce to the case where each edge of $\bar K$ is contained in a 2-simplex. Thus, $m\le 3n$ where $m,n$ are the number of edges and 2-simplices of $\bar K$ respectively. So, in fact, $\bfw(\bar \calF) \le \adelta n$.
\end{remark}

\section{Uniform quasi-surjectivity}\label{sec: quasisurjectivity}

The next proposition shows that continuous quasi-isometries between hyperbolic groups are, in a sense, uniformly quasi-surjective. For this purpose, we will make use of the boundary $\partial G$ of a hyperbolic group $G$. The boundary $\partial G$ is a compact metrizable space which has a finite topological (covering) dimension, $\dim(\partial G)<\infty$ (see \cite{ghys1990groupes}).

\begin{proposition}\label{prop: uniform quasi-onto}
Let $Y$ be a hyperbolic simplicial complex with a proper cocompact group action. Assume that $Y$ is $m$-dimensional and $(m-1)$-connected where $m=\dim(\partial Y)+1$.
Then, there exists $\xi=\xi(Y)$ such that for any group $G$ that acts on $Y$ properly and cocompactly, any $m$-dimensional, $(m-1)$-connected simplicial complex $L$ on which $G$ acts properly and cocompactly, and any continuous $G$-equivariant quasi-isometry $\Phi:L \to Y$ we have $N_{\xi}(\Phi(L))=Y$.
\end{proposition}

\begin{proof}
Since  $\Aut(Y)$ acts cocompactly on $Y$, it suffices to show that the intersection $\Phi(L) \cap B_{\xi}\ne \emptyset$ where $B_{\xi}$ is the ball of radius $\xi$ around some fixed point $x_0$ in $Y$.

The orbit map of the action $G\actson Y$ is a quasi-isometry which induces a boundary homeomorphism $\partial\Phi : \partial G \to \partial Y$.

By Bestvina-Mess \cite{bestvina1991boundary} we have \[m-1=\dim\partial Y = \max\{k \;|\; \cechH ^k(\partial Y) \ne 0\}\] where $\cechH ^k$ is the reduced  $k$-th \v{C}ech cohomology with $\bbZ$ coefficients (see \cite{walsh1981dimension}).

Consider the long exact sequence for the pair $(Y\cup \partial Y, \partial Y)$ in \v{C}ech cohomology
$$\cechH^{m-1}(Y\cup\partial Y) \to \cechH^{m-1} (\partial Y) \to \cechH^m(Y\cup\partial Y,\partial Y).$$
 Note that $\cechH^m(Y\cup\partial Y,\partial Y)\simeq H_c^m(Y)$, where $H_c^m$ is the $m$-th compactly supported cohomology with $\bbZ$ coefficients. In \cref{lem: Cech cohomology} we will prove $\cechH^{m-1}(Y\cup \partial Y)=0$.  
By the exact sequence we get the following embedding
$$\begin{tikzcd} \cechH^{m-1} (\partial Y) \arrow[hookrightarrow,r,"D"] &H_c^m(Y\cup\partial Y,\partial Y)\end{tikzcd}$$
Similarly, for $L$,
$$\begin{tikzcd} \cechH^{m-1} (\partial L) \arrow[hookrightarrow,r,"D"] &H_c^m(L\cup\partial L,\partial L).\end{tikzcd}$$

The map $\Phi\cup \partial \Phi$ is a continuous map of the pairs $ (L \cup \partial L,\partial L) \to (Y \cup \partial Y,\partial Y)$ and by the naturality of the map $D$ in the long exact sequence of the pair, we get the following commutative diagram

\[\begin{tikzcd}
  \cechH ^{m-1}(\partial Y)\arrow[r,hookrightarrow,"D"]\arrow[d,"\simeq","(\partial \Phi)^*"'] &H^m_c(Y)\arrow[d,"\Phi^*"'] \\
 \cechH ^{m-1}(\partial L)   \arrow[hookrightarrow,"D"]{r} &H^m_c(L)
\end{tikzcd}
\]

Fix some $0\ne a\in \cechH ^{m-1}(\partial Y)$, then $D(a)$ is compactly supported, so there exists $\xi>0$ such that $D(a)\in H^n(Y,Y-B_\xi)$. 

Assume for the sake of contradiction that $\Phi(L) \cap B_\xi = \emptyset$. Then, $\Phi^*(D(a))=0$.
On the other hand, $\partial\Phi$ is a homeomorphism, and so $(\partial \Phi)^*$ is an isomorphism. By the commutative diagram we get the following contradiction $$0=\Phi^*(D(a))=D((\partial \Phi)^*(a))\ne 0.$$
\end{proof}

Recall that for a graph $X$ and $d\in\bbN$, the Rips complex $R_d(X)$ is the simplicial complex whose simplices are subsets of the vertex set of $X$ of diameter $\le d$.
If $X$ is $\adelta$-hyperbolic then $R_d(X)$ is contractible for all $d\ge \adelta$.

\begin{lemma}\label{lem: Cech cohomology}
Let $Y$ be as in \cref{prop: uniform quasi-onto} then $\cechH^{m-1}(Y\cup \partial Y)=0$.
\end{lemma}
\begin{proof}
    Consider a contractible Rips complex $R$ for the 1-skeleton of $Y$, note that $Y\subseteq R$ and that $\partial R = \partial Y$.
    Since $Y$ is $(m-1)$-connected, one can build a retraction $r: R^{(m)} \to Y$ inductively, by mapping each cell $\sigma$ of $R - Y$ of dimension $\le m$ to a filling disk of $r|_{\partial \sigma}$ in $Y$. Using the co-compactness of the $G$-action on $R$, the map $r$ can be chosen to be proper, and so it extends to a retraction $r\cup \partial r: R^{(m)}\cup \partial R \to Y\cup \partial Y$. 
    In cohomology it induces an embedding $$\cechH^{m-1} (Y \cup \partial Y) \hookrightarrow \cechH^{m-1}(R^{(m)} \cup \partial R)$$ and so it suffices to show that $\cechH^{m-1}(R^{(m)} \cup \partial R)=0$.
    
    Consider the long exact sequence for the pair $(R \cup \partial R, R^{(m)} \cup \partial R)$
    \begin{equation*}
        \cechH^{m-1}(R \cup \partial R) \to \cechH^{m-1}(R^{(m)}\cup \partial R) \to \cechH^{m}(R \cup \partial R, R^{(m)} \cup \partial R).
    \end{equation*}
     By \cite{bestvina1991boundary}, $R\cup \partial R$ is contractible so $\cechH^{m-1}(R \cup \partial R)=0$. We also have $$\cechH^{m}(R \cup \partial R, R^{(m)} \cup \partial R) = H_c^m(R-R^{(m)})=0$$ since $R-R^{(m)}$ has only cells of dimension $>m$. It follows from the long exact sequence that $\cechH^{m-1}(R^{(m)} \cup \partial R)=0$ as desired.
\end{proof}

\section{Proof of \cref{thm: volume vs complexity}}\label{sec: main proof}
\label{sec: proof of volume vs complexity}

Let $m\in \bbN$. A group $G$ has type $F_m$ if $G$ acts freely and cocompactly on an $(m-1)$-connected $m$-dimensional CW complex $R$.
For $i\le m$ we recall that $\Cx_{i,m}(G)$ is defined to be the minimal number of $G$-orbits of cells of dimension $\le i$ among all such actions on \textbf{simplicial} $R$. 

Torsion-free hyperbolic groups are of type $F$ because they have a free cocompact action on a contractible simplicial complex, namely, a contractible Rips complex. 
Modifying slightly the construction of Rips complexes, see \cite{bestvina1991boundary}, we see that hyperbolic groups are of type $F_\infty$, i.e of type $F_m$ for all $m$. 
Hence, if $G$ is hyperbolic then $\Cx_{i,m}(G)$ is well-defined and finite for all $i,m$ such that $i\le m$. 



\begin{proof}[Proof of \cref{thm: volume vs complexity}]
Let $X$ be a one-ended hyperbolic graph, and let $G\actson X$ be a free cocompact simplicial action. Let $m = \dim(\partial X)+1$. We wish to show $$ \Vol(X/G) \; \asymp_X \; \Cx_{2,m}(G) .$$

\paragraph{Complexity $\ll_X$ volume.} Following \cite{bestvina1991boundary} we construct an $m$-dimensional $(m-1)$-connected complex with a free $G$-action as follows:
Let $Z_d$ be the subcomplex of the $(m+1)$-fold join $X*\dots *X$ of the graph $X$ which consists of those simplices whose vertices are at distance at most $d$ in $X$. Clearly, $G$ acts freely on $Z_d$, and as shown in \cite{bestvina1991boundary}, for large enough $d$ the complex $Z_d$ is $(m-1)$-connected. Let $Z = Z_d^{(m)}$ be the $m$-skeleton of $Z_d$. It is immediate from the construction that $$C_{2,m}(G) \le \Vol_2(Z/G) \le \Vol(Z/G) \asymp_X \Vol(X/G).$$ 

\paragraph{Volume $\ll_X$ complexity.}
Let $Y$ be a contractible Rips complex of $X$. 
Note that $G$ still acts on $Y$ properly and cocompactly, and that $\Vol(X/G) \asymp \Vol(Y/G)\asymp \Vol(Y^{(1)}/G)$. 
Thus, we may assume that $X=Y^{(1)}$. Let $q$ be a globally stable bicombing on $X$. 
By the definition of $C_{2,m}(G)$, there exists an $m$-dimensional, $(m-1)$-connected simplicial complex $L$ with a free and cocompact $G$-action, with $\Vol_2(L/G)=C_{2,m}(G)$. 
Let $K = L^{(2)}$ be its 2-skeleton, and let $\bar K=K/G$ be its quotient. So, $\Vol(\bar K) = \Vol_2(L/G) = C_{2,m}(G)$.
By \cref{accessibility}, there exists a $G$-equivariant map $\Phi:K^{0}\to X^{0}$ and an associated resolution $(\calF,\bfw)$ of $q$ to $K$ satisfying
\begin{equation}\label{upper bound from accessibility}
    \bfw(\bar \calF) \le \adelta\cdot \Vol(\bar K)=\adelta \cdot C_{2,m}(G).
\end{equation}
where $\bar \calF$ is the quotient pattern on $\bar K$.

On the other hand one can extend the map $\Phi$ to a $G$-equivariant continuous quasi-isometry $\Phi:L\to Y$ by induction on the dimension of simplices as follows: First, for every edge $e$ of $L$ define $\Phi|_e$ to be a geodesic connecting $\Phi(e_-),\Phi(e_+)$ chosen in a $G$-equivariant way. For $k>1$ Assume $\Phi$ was defined on $L^{(k-1)}$. Define $\Phi|_{L^{(k)}}$ $G$-equivariantly so that for every $k$-simplex $a$,  $\Phi|_a$ is a filling disk of $\Phi|_{\partial a}$ in $Y$. Such a disk exists, since $Y$ is contractible. In fact, by \cite[Lemma 1.7.A]{gromov1987hyperbolic}\cite[\S4.2 Proposition 9]{ghys1990groupes}, this filling disk can be chosen in some $\adelta$-neighborhood of $\Phi(K^{(1)})$. So we may assume 
\begin{equation}\label{eq: narrow filling disks}
    \Phi(L) \subseteq \calN_\adelta(\Phi(K^{(1)})).
\end{equation}

By \cref{prop: uniform quasi-onto}, $N_\adelta(\Phi(L)) = Y$. 
Therefore, by \cref{eq: narrow filling disks}, there exists $\delta_1=\delta_1(X)$ such that
\begin{equation}
    X \subseteq Y \subseteq N_{\delta_1}(\Phi(K^{(1)})).
\end{equation}
In other words, for all $x\in X$, in the ball of radius $\delta_1$ around $x$, we can find a point $z\in [\Phi(e_-),\Phi(e_+)]$ of some edge $e\in K$. Moreover, we can choose such $z,e$ so that $\Phi(e_-) \ne \Phi(e_+)$, as otherwise the map $\Phi$ would be constant. By \cref{lower bound on contribution of edge} there are $\rho=\rho(X),\lambda = \lambda (X)$ and an edge $f$ in $\calB_\rho(z)$ whose coefficient in $q(e)$ is at least $1/\lambda$.

Let $\alpha$ be the volume of a ball of radius $2(\rho+\delta_1)$ in $X$.
Then, there are at least $\tfrac{1}{\alpha}\Vol(X/G)$ disjoint $(\rho+\delta_1)$-balls in $X/G$. In each of these balls there is an edge $\bar f$ which corresponds to some connector of the pattern $\bar \calF$ whose weight is at least $1/\lambda$. So, $\tfrac{1}{\alpha\lambda}\Vol(X/G) \le \bfw(\bar \calF)$, or simply
\begin{equation}\label{lower bound from surjectivity}
\Vol(X/G)\le \adelta \cdot \bfw(\bar \calF)
\end{equation}
Combining \cref{upper bound from accessibility,lower bound from surjectivity} we get the desired inequality
$$ \Vol(X/G) \le \adelta \cdot C_{2,m}(G).$$
\end{proof}

\bibliographystyle{plain}
\bibliography{biblio}

\begin{thebibliography}{10}

\bibitem{bader2020homology}
Uri Bader, Tsachik Gelander, and Roman Sauer.
\newblock Homology and homotopy complexity in negative curvature.
\newblock {\em Journal of the European Mathematical Society}, 22(8):2537--2571,
  2020.

\bibitem{ballmann1985manifolds}
Werner Ballmann.
\newblock Manifolds of non positive curvature.
\newblock In {\em Arbeitstagung Bonn 1984}, pages 261--268. Springer, 1985.

\bibitem{beeker2016resolutions}
Benjamin Beeker and Nir Lazarovich.
\newblock Resolutions of cat (0) cube complexes and accessibility properties.
\newblock {\em Algebraic \& Geometric Topology}, 16(4):2045--2065, 2016.

\bibitem{beeker2017cubical}
Benjamin Beeker and Nir Lazarovich.
\newblock Cubical accessibility and bounds on curves on surfaces.
\newblock {\em Journal of Topology}, 10(4):1050--1065, 2017.

\bibitem{belolipetsky2010counting}
Mikhail Belolipetsky, Tsachik Gelander, Alexander Lubotzky, and Aner Shalev.
\newblock Counting arithmetic lattices and surfaces.
\newblock {\em Annals of mathematics}, pages 2197--2221, 2010.

\bibitem{bestvina1991boundary}
Mladen Bestvina and Geoffrey Mess.
\newblock The boundary of negatively curved groups.
\newblock {\em Journal of the American Mathematical Society}, 4(3):469--481,
  1991.

\bibitem{bridson2010cofinitely}
Martin~R Bridson, Daniel Groves, Jonathan~A Hillman, and Gaven~J Martin.
\newblock Cofinitely hopfian groups, open mappings and knot complements.
\newblock {\em Groups, Geometry, and Dynamics}, 4(4):693--707, 2010.

\bibitem{brinkmann2000hyperbolic}
Peter Brinkmann.
\newblock Hyperbolic automorphisms of free groups.
\newblock {\em Geometric and Functional Analysis}, 10(5):1071--1089, 2000.

\bibitem{cooper1999volume}
Daryl Cooper.
\newblock The volume of a closed hyperbolic 3-manifold is bounded by times the
  length of any presentation of its fundamental group.
\newblock {\em Proceedings of the American Mathematical Society}, 127(3):941,
  1999.

\bibitem{delzant1995image}
Thomas Delzant.
\newblock L'image d'un groupe dans un groupe hyperbolique.
\newblock {\em Commentarii Mathematici Helvetici}, 70(1):267--284, 1995.

\bibitem{delzant1996decomposition}
Thomas Delzant.
\newblock D{\'e}composition d'un groupe en produit libre ou somme
  amalgam{\'e}e.
\newblock {\em Journal f{\"u}r die reine und angewandte Mathematik},
  470:153--180, 1996.

\bibitem{delzant2013complexity}
Thomas Delzant and Leonid Potyagailo.
\newblock On the complexity and volume of hyperbolic 3-manifolds.
\newblock {\em Israel Journal of Mathematics}, 193(1):209--232, 2013.

\bibitem{dunwoody1985accessibility}
Martin~J Dunwoody.
\newblock The accessibility of finitely presented groups.
\newblock {\em Inventiones mathematicae}, 81(3):449--457, 1985.

\bibitem{gelander2011volume}
Tsachik Gelander.
\newblock Volume versus rank of lattices.
\newblock {\em J. Reine Angew. Math.}, 661:237–248, 2011.

\bibitem{gelander2004homotopy}
Tsachik Gelander et~al.
\newblock Homotopy type and volume of locally symmetric manifolds.
\newblock {\em Duke Mathematical Journal}, 124(3):459--515, 2004.

\bibitem{gelander2019minimal}
Tsachik Gelander and Raz Slutsky.
\newblock On the minimal size of a generating set of lattices in lie groups.
\newblock {\em arXiv e-prints}, pages arXiv--1903, 2019.

\bibitem{gelander2021bounds}
Tsachik Gelander and Paul Vollrath.
\newblock Bounds on systoles and homotopy complexity.
\newblock {\em arXiv preprint arXiv:2106.10677}, 2021.

\bibitem{ghys1990groupes}
{\'E}tienne Ghys.
\newblock Sur les groupes hyperboliques d'apres mikhael gromov.
\newblock {\em Progr. Math.}, 83, 1990.

\bibitem{gromov1982volume}
Michael Gromov.
\newblock Volume and bounded cohomology.
\newblock {\em Publications Math{\'e}matiques de l'IH{\'E}S}, 56:5--99, 1982.

\bibitem{gromov1987hyperbolic}
Mikhael Gromov.
\newblock Hyperbolic groups.
\newblock In {\em Essays in group theory}, pages 75--263. Springer, 1987.

\bibitem{lazarovich2021volume}
Nir Lazarovich.
\newblock Volume vs. complexity of hyperbolic groups.
\newblock {\em arXiv preprint arXiv:2107.13250}, 2021.

\bibitem{mineyev2001straightening}
Igor Mineyev.
\newblock Straightening and bounded cohomology of hyperbolic groups.
\newblock {\em Geometric \& Functional Analysis GAFA}, 11(4):807--839, 2001.

\bibitem{mostow1968quasi}
George~D Mostow.
\newblock Quasi-conformal mappings in $ n $-space and the rigidity of
  hyperbolic space forms.
\newblock {\em Publications Math{\'e}matiques de l'IH{\'E}S}, 34:53--104, 1968.

\bibitem{reznikov1996volumes}
Alexander Reznikov.
\newblock Volumes of discrete groups and topological complexity of homology
  spheres.
\newblock {\em Mathematische Annalen}, 306(3):547--554, 1996.

\bibitem{rips1995canonical}
Eliyahu Rips and Zlil Sela.
\newblock Canonical representatives and equations in hyperbolic groups.
\newblock {\em Inventiones mathematicae}, 120(1):489--512, 1995.

\bibitem{sela1997structure}
Zlil Sela.
\newblock Structure and rigidity in (gromov) hyperbolic groups and discrete
  groups in rank 1 lie groups ii.
\newblock {\em Geometric \& Functional Analysis GAFA}, 7(3):561--593, 1997.

\bibitem{stark2018hyperbolic}
Emily Stark and Daniel~J Woodhouse.
\newblock Hyperbolic groups that are not commensurably co-hopfian.
\newblock {\em International Mathematics Research Notices}, 2018.

\bibitem{sykiotis2018complexity}
Mihalis Sykiotis.
\newblock Complexity volumes of splittable groups.
\newblock {\em Journal of Algebra}, 503:409--432, 2018.

\bibitem{thurston1979geometry}
William~P Thurston.
\newblock {\em The geometry and topology of three-manifolds}.
\newblock Princeton University Princeton, NJ, 1979.

\bibitem{walsh1981dimension}
John~J Walsh.
\newblock Dimension, cohomological dimension, and cell-like mappings.
\newblock In {\em Shape theory and geometric topology}, pages 105--118.
  Springer, 1981.

\bibitem{wang1994covering}
Shicheng Wang and Ying-Qing Wu.
\newblock Covering invariants and cohopficity of 3-manifold groups.
\newblock {\em Proceedings of the London Mathematical Society}, 3(1):203--224,
  1994.

\bibitem{wang1999covering}
Shicheng Wang and F~Yu.
\newblock Covering degrees are determined by graph manifolds involved.
\newblock {\em Commentarii Mathematici Helvetici}, 74(2):238--247, 1999.

\end{thebibliography}
\end{document}